\documentclass[11pt, a4paper]{amsart}
\usepackage{amssymb}
\usepackage{amsmath}
\usepackage{amssymb}
\usepackage[abbrev,backrefs]{amsrefs}
\usepackage{amsthm}
\usepackage{bbm}
\usepackage{bm}
\usepackage{mathtools}
\usepackage{mathrsfs}
\usepackage[T1]{fontenc}
\usepackage[usenames,dvipsnames,svgnames,table]{xcolor}
\usepackage{esint}
\allowdisplaybreaks
\usepackage{enumitem}
\setenumerate{label={\rm (\alph{*})}}
\usepackage[colorlinks=true,citecolor=black,urlcolor=black,
linkcolor=black]{hyperref}

\usepackage{combelow}

\usepackage{graphicx}
\usepackage{tikz}

\usepackage{geometry}
\makeatletter
\newcommand*\bigcdot{\mathpalette\bigcdot@{.5}}
\newcommand*\bigcdot@[2]{\mathbin{\vcenter{\hbox{\scalebox{#2}{$\m@th#1\bullet$}}}}}
\makeatother

\numberwithin{equation}{section}

\newtheorem{theorem}{Theorem}[section]
\newtheorem{lemma}[theorem]{Lemma}

\newtheorem{definition}[theorem]{Definition}

\numberwithin{equation}{section}

\theoremstyle{remark}
\newtheorem{remark}[theorem]{Remark}

\usepackage{booktabs}

\normalsize

\def\XXint#1#2#3{{\setbox0=\hbox{$#1{#2#3}{\int}$ }
\vcenter{\hbox{$#2#3$ }}\kern-.6\wd0}}

\newcommand{\bv}{\operatorname{BV}}

\newcommand{\dif}{\operatorname{d}\!}

\newcommand{\R}{\mathbb{R}}
\newcommand{\A}{\mathbb{A}}

\newcommand{\locc}{\operatorname{loc}}

\newcommand{\sobo}{\operatorname{W}}

\newcommand{\lebe}{\operatorname{L}}

\renewcommand{\leq}{\leqslant}

\newcommand{\id}{\operatorname{Id}}

\newcommand{\lin}{\operatorname{Lin}}

\usepackage{tikz-cd}
\newcommand{\calC}{\mathcal{C}}
\newcommand{\calL}{\mathcal{L}}
\newcommand{\smzs}{\setminus\{0\}}

\begin{document}
\title[Elliptic Systems with $L^1$ Data]{A Note on Estimates for Elliptic Systems with $L^1$ Data}
\author[B. Rai\cb{t}\u{a}]{Bogdan Rai\cb{t}\u{a}}
\author[D. Spector]{Daniel Spector}
\begin{abstract}
In this paper we give necessary and sufficient conditions on the compatibility of a $k$th order homogeneous linear elliptic differential operator $\A$ and differential constraint $\calC$ for solutions of
\begin{align*}
		\A u=f\quad\text{subject to}\quad\calC f=0\quad\text{ in }\R^n
\end{align*}
to satisfy the estimates
\begin{align*}
\|D^{k-j}u\|_{\lebe^{\frac{n}{n-j}}(\R^n)}\leq c\|f\|_{\lebe^1(\R^n)}
\end{align*}
for $j\in \{1,\ldots,\min\{k,n-1\}\}$ and
\begin{align*}
\|D^{k-n}u\|_{\lebe^{\infty}(\R^n)}\leq c\|f\|_{\lebe^1(\R^n)}
\end{align*}
when $k\geq n$.  
\end{abstract}
\maketitle
\section{Introduction}
Let $f \in \lebe^1(\R^n,\R^n)$ and consider the problem of finding estimates for $u:\R^n \to \R^n$ which satisfies
\begin{align}\label{Laplace}
		-\Delta u=f \text{ in }\R^n.
\end{align}
While it is well-known that without further assumptions no inequalities of the form
\begin{align}
\|\nabla u\|_{\lebe^{\frac{n}{n-1}}(\R^n)}&\leq c\|f\|_{\lebe^1(\R^n)} \label{estimate1}
\quad\text{and}\\
\|u\|_{\lebe^{\frac{n}{n-2}}(\R^n)}&\leq c\|f\|_{\lebe^1(\R^n)} \label{estimate2}
\end{align}
are possible\footnote{Take $f$ to be a Dirac delta in any of its components.}, in the pioneering papers \cite{BourgainBrezis2004, BB07} J. Bourgain and H. Brezis have shown that under the additional constraint $\operatorname*{div} f=0$, \eqref{estimate1} and \eqref{estimate2} are indeed valid.  Precisely, their Theorem 2 in \cite{BB07} establishes the validity of \eqref{estimate1} and \eqref{estimate2} in three or more dimensions, while their Theorem 3 shows that in two dimensions one has \eqref{estimate1} and 
\begin{align}\label{estimate2'}
\|u\|_{\lebe^{\infty}(\R^n)}&\leq c\|f\|_{\lebe^1(\R^n)}.\end{align}
A simple proof of these estimates was subsequently given by J. Van Schaftingen in \cite{VS-CR}, who went on in \cite{VS} to show that the estimates \eqref{estimate1} and \eqref{estimate2} actually hold under very general assumptions on $f$ which we discuss in more detail in the sequel.

The purpose of this paper is to address the question of necessary and sufficient conditions to obtain estimates in this spirit for solutions of the elliptic system
\begin{align}\label{PDE}
		\A u=f \quad\text{subject to}\quad \calC f=0\quad\text{ in }\R^n
\end{align}
for $\A : C^\infty_c(\R^n,V) \to C^\infty_c(\R^n,E)$ a $k$th order homogeneous linear elliptic differential operator,  $\calC: C^\infty_c(\R^n,E) \to C^\infty_c(\R^n,F)$ an $l$th order  homogeneous linear differential operator, and $V,E,F$ finite dimensional inner product spaces.  In particular, our work builds upon the foundational results of J. Van Schaftingen \cite{VS} to give a complete characterization of the conditions on $\A$ and $\calC$ such that the estimates
\begin{align}\label{Lebesgueestimates}
\|D^{k-j}u\|_{\lebe^{\frac{n}{n-j}}(\R^n)}\leq c\|f\|_{\lebe^1(\R^n)}
\end{align}
hold for $j\in \{1,\ldots,\min\{k,n-1\}\}$ or
\begin{align}\label{Linftyestimate}
		\|D^{k-n}u\|_{\lebe^{\infty}(\R^n)}\leq c\|f\|_{\lebe^1(\R^n)}
	\end{align}
if $k \geq n$.

To this end let us recall what can already be said in light of the literature \cite{BourgainBrezis2004,BB07,VS,BVS,R_diff}.  We first consider the case of the estimates \eqref{Lebesgueestimates}.  Moving beyond the the preceding inequalities of J. Bourgain and H. Brezis \cite{BourgainBrezis2004,BB07}, J. Van Schaftingen's work (see Proposition 8.7 in \cite{VS}) shows that for
\begin{align*}
\A&=(-\Delta)^{k/2}
\end{align*}
one has \eqref{Lebesgueestimates} if and only if $\calC$ is cocanceling:
\begin{align}\label{eq:cocanc}
\bigcap_{\xi\in\R^n\smzs}\ker\calC(\xi)=\{0\}.
\end{align}

This notion of cocanceling utilizes the convention that the homogeneous linear differential operator $\calC$, which has a representation as
\begin{align*}
\calC f=\sum_{|\alpha|= l}C_\alpha \partial^\alpha f,\quad\text{for }f\colon\R^n\rightarrow E,
\end{align*}
for some $l \in \mathbb{N}_0$ and coefficients $C_\alpha\in\lin(E,F)$, can be viewed via its image under the Fourier transform, which is a matrix-valued polynomial defined by 
\begin{align*}
	\calC(\xi)=\sum_{|\alpha|= l}C_\alpha \xi^\alpha\in\lin(E,F),\quad\text{for }\xi\in\R^n.
	\end{align*}
The essence of the condition \eqref{eq:cocanc} is found in the proof of Proposition 2.1 in \cite{VS}, which shows
\begin{align}\label{cocanc_char}
		\bigcap_{\xi\in\R^n}\ker \calC(\xi)=\left\{e\in E\;\colon\; \calC(\delta_0 e)=0\right\}.
\end{align}
In particular, the heuristic principle concerning the failure of the inequalities \eqref{estimate1}, \eqref{estimate2}, and \eqref{estimate2'} precisely when $f$ contains a Dirac mass in one of its components is captured by the necessity and sufficiency of \eqref{eq:cocanc} via the equivalence \eqref{cocanc_char}.

While cocancelation gives the complete picture in characterizing the estimates \eqref{Lebesgueestimates} for $\A=(-\Delta)^{k/2}$, it ceases to be necessary when one has assumes additional structure on $\A$.  In particular, with no differential constraint $\calC$, J. Van Schaftingen has shown that the inequality \eqref{Lebesgueestimates} holds whenever $\A$ is canceling:
\begin{align}\label{eq:canc}
\bigcap_{\xi\in\R^n\smzs}\mathrm{im\,}\A(\xi)=\{0\}.
\end{align}
Here again we view $\A$ via its image under the Fourier transform, while this set also has an equivalent representation in terms of fundamental solutions of operator $\A$, which follows from the proof of Lemma 2.5 in \cite{R_diff}:
\begin{align}\label{canc_char}
		\bigcap_{\xi\in\R^n\setminus\{0\}}\mathrm{im\,}\A(\xi)=\left\{e\in E\;\colon\; \A u=\delta_0 e\text{ for some }u\in\lebe^1_{\locc}(\R^n,V)\right\}.
	\end{align} 
The connection of the conditions \eqref{cocanc_char} and \eqref{canc_char} here emerges, that for a canceling operator one can find a cocanceling annihilator and therefore apply the preceding analysis.

However, while $\calC f=0$ for some cocanceling operator $\calC$ or $\A$ is canceling is sufficient to imply the validity of \eqref{Lebesgueestimates} for $j\in \{1,\ldots,\min\{k,n-1\}\}$, neither is necessary.  Indeed, the first result of this paper is
\begin{theorem}\label{thm1}
	Let $\A,\,\calC$ be homogeneous linear differential operators on $\R^n$ from $V$ to $E$ and from $E$ to $F$, respectively. Suppose that $\A$ is elliptic and has order $k\in\mathbb{N}$. Consider the system
	\begin{align}\label{eq:system}
		\A u=f\quad\text{subject to}\quad\calC f=0\quad\text{ in }\R^n.
	\end{align}
	Let $j=1,\ldots,\min\{k,n-1\}$. Then the estimate for $u\in C^\infty_c(\R^n,V),\,f\in C^\infty_c(\R^n,E)$ satisfying \eqref{eq:system}
	\begin{align}\label{eq:estimate}
		\|D^{k-j}u\|_{\lebe^{\frac{n}{n-j}}(\R^n)}\leq c\|f\|_{\lebe^1(\R^n)}
	\end{align}
	holds if and only if
		\begin{align}\tag{CC}\label{eq:compensated_canc}
			\bigcap_{\xi\in\R^n\smzs}\mathrm{im\,}\A(\xi)\cap\bigcap_{\xi\in\R^n\smzs}\ker\calC(\xi)=\{0\}.
		\end{align}
\end{theorem}
This result is in the spirit of Theorem 7.1 in \cite{VS}, where the author introduces a notion of partially canceling operators.   The idea there, which we build upon here, is that while neither \eqref{eq:cocanc} nor \eqref{eq:canc} is empty, the two are disjoint.  While in \cite{VS} J. Van Schaftingen treats the case $\calC=T \in \lin(E,F)$ is a linear map from $E$ to $F$, our result handles the case of homogeneous differential operators, which is reduced to his framework by our Lemma \ref{lem:alg_reduction} below.  
\begin{remark}
 One can consider more general differential operators $\calC$, which are not homogeneous, as we make precise below in Definition~\ref{def:dif_op}. In particular, the estimate \cite[Thm.~1.4]{VS} for cocanceling operators is extended to such operators (see Lemma~\ref{lem:alg_reduction} and Remark~\ref{rk:cocanc} below). Fractional estimates, e.g., $u\in\dot{\sobo}{^{k-s,n/(n-s)}}$ for $0<s<n$ and $s\leq k$, are also possible, by merging our ideas with \cite[Sec.~8]{VS}.
\end{remark}

Theorem \ref{thm1}, for example, shows that with $V=\R^3$, $E=\R^4$, $F=\R$ and 
\begin{align*}
\A &\coloneqq (\operatorname*{div}, \operatorname*{curl}),\\
\calC f &\coloneqq \partial_1 f_1+\partial_2 f_2+\partial_3f_3,
\end{align*}
solutions to \eqref{PDE} admit the estimate \eqref{eq:system} with $j=k=1$ for any $n \geq 2$. Notice that $\calC$ is not cocanceling because its kernel contains $\delta_0e_4$ (which means that \eqref{eq:cocanc} contains the vector $e_4$), while $\A$ is not canceling, as its image contains $\delta_0e_1$ (which means that \eqref{eq:canc} contains the vector $e_1$).  Here $\{e_j\}_{j=1}^4$ is the standard orthonormal basis of $\R^4$.  

Returning to the question of the validity of the embedding \eqref{Linftyestimate} for $k\geq n$, P. Bousquet and J. Van Schaftingen \cite{BVS} have shown that such an inequality holds whenever $\A$ is canceling, while the first author has proved in \cite{R_diff} that this holds if and only if $\A$ is weakly canceling:
\begin{align}\label{eq:w_canc}
\int_{\mathbb{S}^{n-1}}\A(\xi)^{-1}e\otimes^{k-n}\xi\dif\mathscr{H}^{n-1}(\xi)=0\quad\text{for all }e\in\bigcap_{\xi\in\R^n\smzs}\mathrm{im\,}\A(\xi),
\end{align}
where $v\otimes^j\xi\coloneqq v\otimes \xi\otimes\ldots\otimes \xi$, where the outer product is taken $j$ times.

The question of the validity of such an inequality for $\calC f=0$, $\calC$ cocanceling, has not thus far been explicitly addressed, save the new various compatibility conditions that we introduce.  
In this regime we show
\begin{theorem}\label{thm2}
	Let $\A,\,\calC$ be homogeneous linear differential operators on $\R^n$ from $V$ to $E$ and from $E$ to $F$, respectively. Suppose that $\A$ is elliptic and has order $k\geq n$. Then the estimate for $u\in C^\infty_c(\R^n,V),\,f\in C^\infty_c(\R^n,E)$ satisfying \eqref{eq:system}
	\begin{align}\label{eq:estimate_infty}
		\|D^{k-n}u\|_{\lebe^{\infty}(\R^n)}\leq c\|f\|_{\lebe^1(\R^n)}
	\end{align}
	holds if and only if
		\begin{align}\tag{CWC}\label{eq:compensated_w_canc}
\int_{\mathbb{S}^{n-1}}\A(\xi)^{-1}e\otimes^{k-n}\xi\dif\mathscr{H}^{n-1}(\xi)=0\text{ for }e\in\bigcap_{\xi\in\R^n\smzs}\mathrm{im\,}\A(\xi)\cap\ker\calC(\xi).
		\end{align}
\end{theorem}
Theorem \ref{thm2} implies that for even dimensions solutions of \eqref{PDE} with
\begin{align*}
\A &= (-\Delta)^{n/2}\\
\calC &= \operatorname*{div}
\end{align*}
are bounded, which is a higher dimensional analogue of the estimate \eqref{estimate2'} to the equation \eqref{Laplace} due to J. Bourgain and H. Brezis when $n=2$ (naturally the order of the equation must be modified to achieve an $\lebe^\infty$ embedding).  More generally, this applies for any cocanceling operator $\calC$, while again one can construct $\calC$ which are not cocanceling and $\A$ which are not weakly canceling for which our result holds, e.g., in $\R^4$ with $V=\R^2$, $E=\R^3$, $F=\R$ and
\begin{align*}
\A &\coloneqq ((\partial_1^4+\partial_2^4) u_1, \partial_3^4 u_2,\partial_4^4u_2)\\
\calC f&\coloneqq \partial_1 f_1+\partial_2f_2.
\end{align*}
For this example, one computes explicitly that $e_3\in\ker\calC(\xi)$ for all $\xi\neq0$, so that $\calC$ is not cocanceling.  On the other hand, $e_1\in\mathrm{im\,}\A(\xi)$ for all $\xi\neq0$ and that $M_\A e_1\neq0$, so that $\A$ is not weakly canceling (see \eqref{eq:w_canc}). 

The emergence of the integral over the sphere in \eqref{eq:compensated_w_canc} and \eqref{eq:w_canc} stems from the convolution formula proved in \cite[Sec.~3]{R_diff} building on \cite[Thm.~7.1.20]{HormI}, namely 
\begin{align}\label{eq:conv}
	D^{k-n}u=K\ast \A u\quad\text{for}\quad u\in C^\infty_c(\R^n,V),\quad\text{where}\quad K=H_0+\log|\cdot|M_\A
\end{align}
for
\begin{align}\label{eq:M_A}
M_\A e\coloneqq\int_{\mathbb{S}^{n-1}}\A^\dagger(\xi)e\otimes^{k-n}\xi\dif\mathscr{H}^{n-1}(\xi)\quad \text{for }e\in E,
\end{align}
where $\A^\dagger(\xi)\coloneqq(\A^*(\xi)\A(\xi))^{-1}\A^*(\xi)$.
 Here $H_0\in C^\infty(\R^n\setminus\{0\},\lin(E,V))$ is zero-homogeneous and we consider a renormalization of the Fourier transform such that the constants are correct.

\section{Proofs}
\begin{definition}\label{def:dif_op}
	We will \emph{only} work with vectorial partial differential operators on $\R^n$ which have real constant coefficients and are homogeneous in each entry. To make this precise, an operator $\calC$ on $\R^n$ from $E$ to $F$ can be written as 
	$$
	(\calC(\xi)e)_j=\langle C_j(\xi),e\rangle,\quad  e\in E,\,\xi\in\R^n,\quad j=1,\ldots,\dim F,
	$$
	where $C_j$ are $E$-valued \emph{homogeneous} polynomials. A homogeneous operator $\calC$ will correspond to all $C_j$ being homogeneous of the same degree, say $l$, in which case we can write 
	$$
	\calC(\xi)=\sum_{|\beta|=l} \xi^\beta \calC_\beta,
	$$
	which is a $\lin(E,F)$-valued homogeneous polynomial (here $\calC_\beta\in\lin(E,F)$).
\end{definition}
The following algebraic reduction lemma will play an important role in establishing sufficiency of either \eqref{eq:compensated_canc} or \eqref{eq:compensated_w_canc} for the claimed estimates.
\begin{lemma}\label{lem:alg_reduction}
	Let $\calC$ be a linear differential operator on $\R^n$ from $E$ to $F$, as given by Definition~\ref{def:dif_op}. Then there exists a \emph{homogeneous} differential operator $\tilde\calC$ on $\R^n$ from $E$ to another vector space $\tilde F$ such that
	\begin{align*}
		\ker \tilde{\calC}(\xi)=\ker \calC(\xi)\quad \text{for all }\xi\in\R^n
	\end{align*}
	and
	\begin{align*}
		\{f\in C^\infty_c(\R^n,E)\colon \tilde\calC f=0 \}=		\{f\in C^\infty_c(\R^n,E)\colon \calC f=0 \}.
	\end{align*}
\end{lemma}
\begin{proof}
We write $(\calC(\xi)e)_j=\langle \calC_j(\xi),e\rangle$ for the rows of $\calC$, $j=1,\ldots,\dim F$. These define (scalar) differential operators on $\R^n$ from $E$ to $\R$. Let now $d_j$ be the degree of $C_j$ and consider an integer $l\geq \max\{d_j\}_{j=1}^{\dim F}$. Define the differential operators
\begin{align*}
	\tilde C_j(\xi)\coloneqq C_j(\xi)\otimes ^{l-d_j}\xi,\quad\text{ so that }\quad\tilde C_j f=D^{l-d_j}C_j f\text{ for }f\in C^\infty_c(\R^n,E).
\end{align*}
Defining $\tilde\calC$ to be the collection of all the equations given by $\tilde C_j$, it is immediate to see that the inclusions ``$\supset$'' hold.

Conversely, if $f\in C^\infty_c(\R^n,E)$ is such that $\tilde \calC f=0$, we have from the above formula that $D^{l-d_j} C_j f=0$ for all $j$. Since $C_j f\in C^\infty_c(\R^n)$, we conclude that $\calC f=0$. The other conclusion follows in a similar way, using the fact that $\otimes^{l-d_j}\xi\neq0$ whenever $\xi\neq 0$.
\end{proof}
	\begin{remark}\label{rk:cocanc}
	A first relevant consequence of Lemma~\ref{lem:alg_reduction} is that the estimate for cocanceling operators \cite[Thm.~1.4]{VS} holds for a larger class of (inhomogeneous) operators, as given by Definition~\ref{def:dif_op}. In this case, cocancellation would be defined the same as in \cite[Def.~1.2]{VS}.
\end{remark}
We can now proceed with the proof of  Theorem~\ref{thm1}.
\begin{proof}[Proof of necessity of \eqref{eq:compensated_canc}]
	Suppose that condition~\eqref{eq:compensated_canc} fails, so there exists $0\neq e\in\mathrm{im\,}\A(\xi)\cap\ker\calC(\xi)$ for all $\xi\neq0$.
		 Then $\calC (\delta_0e)=0$ and there exists $u\in\lebe^1_{\locc}(\R^n,V)$ such that $\A u=\delta_0e$ (see the proofs of \cite[Prop.~2.1]{VS} and \cite[Lem.~2.5]{R_diff}). In particular, $u$ is admissible for the estimate \eqref{eq:estimate}. We recall from \cite[Lem.~2.1]{BVS} that for $v\in C^\infty_c(\R^n,V)$, its derivatives can be retrieved from $\A v$ by convolution. In particular, if $j\leq\min\{k,n-1\}$, we have that $D^{k-j}v=H_{j-n}\ast \A v$, where $H_{j-n}\in C^\infty(\R^n\setminus\{0\})$ is a $(j-n)$-homogeneous kernel. It follows that $D^{k-n}u=H_{j-n}e$, which contradicts the estimate unless $e=0$.
\end{proof}

\begin{proof}[Proof of sufficiency of \eqref{eq:compensated_canc}]
	From \cite[Sec.~4.2]{VS}, we know that there exists a homogeneous linear differential operator $L(D)$ such that $\ker L(\xi)=\mathrm{im\,}\A(\xi)$ for all $\xi\neq 0$. In particular, condition~\eqref{eq:compensated_canc} implies that the operator $\calL\eqqcolon(L(D),\calC)$ is cocanceling, so that by Remark~\ref{rk:cocanc} we have the estimate
	\begin{align}\label{eq:est1}
		\|\A u\|_{\dot{\sobo}{^{-1,\frac{n}{n-1}}}(\R^n)}=\|f\|_{\dot{\sobo}{^{1,n}}(\R^n)^*}\leq c\|f\|_{\lebe^1(\R^n)}
	\end{align}
	for $u\in C^\infty_c(\R^n,V)$, $f\in C^\infty_c(\R^n,E)$ satisfying \eqref{eq:system}. We then write in Fourier space
	\begin{align*}
		\widehat{D^{k-1}u}(\xi)=|\xi|\A^\dagger(\xi)\frac{\widehat{\A u}(\xi)}{|\xi|}\otimes^{k-1}\xi,
	\end{align*}
	so that the H\"ormander-Mihlin multiplier theorem implies that 
	\begin{align}\label{eq:est2}
		\|D^{k-1}u\|_{\lebe^{\frac{n}{n-1}}(\R^n)}\leq c\left\|\mathscr{F}^{-1}\left(\frac{\widehat{\A u}(\xi)}{|\xi|}\right)\right\|_{\lebe^{\frac{n}{n-1}}(\R^n)}=c\|\A u\|_{\dot{\sobo}{^{-1,\frac{n}{n-1}}}(\R^n)}.
	\end{align}
	Collecting estimates \eqref{eq:est1} and \eqref{eq:est2}, we obtain the desired inequality for $j=1$. The inequalities for $j=2,\ldots,\min\{k,n-1\}$ follow by iteration of the Sobolev inequality.
\end{proof}
It remains to prove Theorem~\ref{thm2}. Recall the definition \eqref{eq:M_A}.
\begin{proof}[Proof of necessity of \eqref{eq:compensated_w_canc}] 
	Suppose that condition~\eqref{eq:compensated_w_canc} fails, so there exists $0\neq e\in\mathrm{im\,}\A(\xi)\cap\ker\calC(\xi)$ for all $\xi\neq0$ such that $M_\A e\neq0$.
	 Then $\calC (\delta_0e)=0$ and there exists $u\in\lebe^1_{\locc}(\R^n,V)$ such that $\A u=\delta_0e$ (see the proofs of \cite[Prop.~2.1]{VS} and \cite[Lem.~2.5]{R_diff}). In particular, $u$ is admissible for the estimate \eqref{eq:estimate}. By \eqref{eq:conv},
\begin{align*}
\|D^{k-n}u\|_{\lebe^\infty}\geq|\|H_0e\|_{\lebe^\infty}-\|\log|\cdot|M_\A e\|_{\lebe^{\infty}}|,
\end{align*}
which is clearly infinite (near 0) since $M_\A e\neq0$ and $H_0$ is bounded.
\end{proof}
To prove sufficiency of \eqref{eq:compensated_w_canc}, we employ a streamlined variant of \cite[Lem.~3.1]{R_diff}, which relies on \cite[Lem.~2.2]{BVS} and \cite[Lem.~2.5]{VS} (see also \cite{VS,BB07,VS_BMO,Mazya_JEMS,BousquetMironescu}).
\begin{lemma}\label{lem:BVS_var}
Let $\calL$ be a linear differential operator on $\R^n$ from $E$ to $F$ as given in Definition~\ref{def:dif_op} and $M\in\lin(E,W)$. Suppose that $M\left(\bigcap_{\xi\neq0}\ker \calL(\xi)\right)=\{0\}$. Then for all $w\in W$ with $|w|=1$ we have
\begin{align*}
\left|\int_{\R^n}\langle\log|x|M^*w,f(x)\rangle\dif x\right|\leq c\|f\|_{\lebe^1(\R^n)}\quad\text{for }f\in C^\infty_c(\R^n,E)\text{ such that }\calL f=0.
\end{align*}
\end{lemma}
\begin{proof}
	By Lemma~\ref{lem:alg_reduction}, we can assume that $\calL$ is homogeneous, say of order $l$, which we write as $\calL =\sum_{|\beta|=l}\calL_\beta\partial^\beta$, where $\calL_\beta\in\lin(E,F)$.
Since $(\xi^\beta)_{|\beta|=l}$ is a basis for homogeneous polynomials of degree $l$, we have that $e\in\ker\calL(\xi)$ for all $\xi\neq0$ is equivalent with $e$ lying in the kernel of the map $T\colon w\mapsto(\calL_\beta e)_{|\beta|=l}$. By assumption, we have that $\mathrm{im\,}M^*\cap\bigcap_{\xi\in\mathbb{S}^{n-1}}\ker\calL(\xi)=\{0\}$, hence the restriction of $T$ to $\mathrm{im\,}M^*$ is injective. Equivalently, this restriction is left-invertible, so there exist linear maps $K_\alpha\in\lin(F,\mathrm{im\,}M^*)$ such that
\begin{align*}
\sum_{|\beta|=l}K_\beta \calL_\beta\restriction_{\mathrm{im\,}M^*}=\id_{\mathrm{im\,}M^*}.
\end{align*}
Define now the matrix-valued field
\begin{align*}
P(x)\coloneqq\sum_{|\beta|=l}\dfrac{x^\beta}{\beta!}K_\beta^*,
\end{align*}
which is essentially a right-inverse (integral) of $\calL^*$, as
\begin{align}\label{eq:right_FS}
\mathcal{L}^*P=\sum_{|\beta|=l}\mathcal{L}^*_\beta\partial^\beta P=\sum_{|\beta|=l}\mathcal{L}^*_\beta K^*_{\beta}=\id_{\mathrm{im\,}M^*}.
\end{align}
Writing $\varphi\coloneqq\log|x|M^*w$ and integrating by parts using $\calL f=0$, we have that
\begin{align*}
\left|\int_{\R^n}\langle\varphi,f\rangle\dif x\right|=\left|\int_{\R^n}\langle[\mathcal{L}^*P]\varphi,f\rangle\dif x\right|
=\left|\int_{\R^n}\langle[\mathcal{L}^*P]\varphi-\mathcal{L}^*[P\varphi],f\rangle\dif x\right|.
\end{align*}
We then note that:
\begin{align*}
[\mathcal{L}^*P]\varphi-\mathcal{L}^*[P\varphi]=[\mathcal{L}^*P]\varphi-[\mathcal{L}^*P]\varphi-\sum_{j=1}^lB_j(D^j\varphi,D^{l-j}P),
\end{align*}
where $B_j$ are bilinear pairings on finite dimensional spaces that depend on $\mathcal{L} $ only. Note that $|D^{j}\varphi|\leq c|\cdot|^{-j}$ and $|D^{l-j}P|\leq c|\cdot|^j$ for $j= 1,\ldots l$ (here it is crucial that $j\geq 1$), so the conclusion follows. 
\end{proof}

\begin{proof}[Proof of sufficiency of \eqref{eq:compensated_w_canc}] 
	By \eqref{eq:conv}, the triangle inequality, and Young's convolution inequality, we have that
	\begin{align*}
	\|D^{k-n}u\|_{\lebe^\infty}&\leq c\left(\|H_0\ast f\|_{\lebe^\infty}+\|\log|\cdot|\ast[M_\A f]\|_{\lebe^\infty}\right)\\
	&\leq c\left(\|H_0\|_{\lebe^\infty}\|f\|_{\lebe^1}+\|\log|\cdot|\ast[M_\A f]\|_{\lebe^\infty}\right),
	\end{align*}
	so it suffices to prove that 
	\begin{align*}
	\|\log|\cdot|\ast[M_\A f]\|_{\lebe^\infty}\leq c\|f\|_{\lebe^1},
	\end{align*}
	for $u\in C^\infty_c(\R^n,V)$. Equivalently, it remains to show that for all $v\in V$, $\eta\in\R^n$ of unit length, we have that
	\begin{align*}
	\left|\int_{\R^n}\langle\log|y|v\otimes^{k-n}\eta,M_\A f(x-y)\rangle\dif y\right|\leq c\|f\|_{\lebe^1}\quad\text{for }\mathscr{L}^n\text{-a.e. }x\in\R^n,
	\end{align*}
	which follows from Lemma~\ref{lem:BVS_var} with $M=M_\A$, $\calL=(\calC,L(D))$, $W=V\odot^{k-n}\R^n$, and $w=v\otimes^{k-n}\eta$ (note that the estimate of Lemma~\ref{lem:BVS_var} is translation invariant). Here we wrote, as in the proof of sufficiency of \eqref{eq:compensated_canc} for Theorem~\ref{thm1}, $L(D)$ for an exact annihilator of $\A$, by which we mean $\ker L(\xi)=\mathrm{im\,}\A(\xi)$ for all $\xi\neq 0$ (see \cite[Sec.~4.2]{VS}). With this notation, condition~\eqref{eq:compensated_w_canc} is equivalent with the assumption on $M$ and $\calL$ in Lemma~\ref{lem:BVS_var}.
\end{proof}

\section*{Acknowledgements}
B.R. thanks the National Center for Theoretical Sciences of Taiwan for the funding  and the National Chiao Tung University for their hospitality during his visit, when most of the present work was completed. This project has received funding from the European Research Council (ERC) under the European Union's Horizon 2020 research and innovation programme under grant agreement No 757254 (SINGULARITY).

D.S. is supported in part by the Taiwan Ministry of Science and Technology under research grant 107-2115-M-009-002-MY2.

\bibliographystyle{amsplain}

\end{document}